\documentclass{amsart}
\usepackage[all,cmtip]{xy}
\usepackage{pb-diagram,pb-xy}
\usepackage{amsmath,amscd,amsthm,amsfonts, amssymb,amsxtra}
\CDat

\SelectTips{cm}{}
\newtheorem{thm}{Theorem}[section]  
\newtheorem*{un-no-thm}{Theorem}
\newtheorem{lem}[thm]{Lemma}         
\newtheorem{prop}[thm]{Proposition}  

\newtheorem{bigthm}{Theorem}

\newtheorem{bigcor}[bigthm]{Corollary}

\newtheorem{bigprop}[bigthm]{Proposition}

\newtheorem*{mainthm}{Theorem}

\theoremstyle{definition}
\newtheorem{defn}[thm]{Definition}   

\theoremstyle{definition}

\theoremstyle{definition}

\theoremstyle{remark}
\newtheorem{rem}[thm]{Remark}        

\newtheorem*{acks}{Acknowledgements}

\newtheorem{ex}[thm]{Example}

\begin{document}
\title{Poincar\'e duality and Periodicity}
\date{\today}
\author{John R. Klein}
\address{Dept.\ of Mathematics, Wayne State University, Detroit, MI 48202}
\email{klein@math.wayne.edu}
\author{William Richter}
\address{Dept.\ of Mathematics, Northwestern University, Evanston, IL 60208}
\email{richter@math.northwestern.edu}
\thanks{The first author is partially supported by the NSF}
\begin{abstract}
We construct periodic families of Poincar\'e complexes, partially
solving a question of Hodgson, and infinite families of Poincar\'e
complexes whose top cell falls off after one suspension but which fail
to embed in a sphere of codimension one.  We give a homotopy theoretic
description of the four-fold periodicity in knot cobordism.
\end{abstract}
\maketitle
\setlength{\parindent}{15pt}
\setlength{\parskip}{1pt plus 0pt minus 1pt}
\def\bdot{\bold .}
\def\Top{\bold T\bold o \bold p}
\def\Sp{\bold S\bold p}
\def\vo{\varOmega}
\def\smsh{\wedge}
\def\^{\wedge}
\def\flush{\flushpar}
\def\id{\text{\rm id}}
\def\dbslash{/\!\! /}
\def\codim{\text{\rm codim\,}}
\def\:{\colon}
\def\holim{\text{holim\,}}
\def\hocolim{\text{hocolim\,}}
\def\cal{\mathcal}
\def\Bbb{\mathbb}
\def\bold{\mathbf}
\def\simtwohead{\,\, \hbox{\raise1pt\hbox{$^\sim$} \kern-13pt $\twoheadrightarrow \, $}}
\def\codim{\text{\rm codim\,}}
\def\stableto{\mapstochar \!\!\to}

 \let\Sec=\S
\def\S{\Sigma}

\section{Introduction} 
Let $X^n$ be a finite oriented Poincar\'e complex of dimension 
$\ge 3$.  We may suppose~\cite[Thm.~2.4]{Wall3} that 
$X = K \cup_{\alpha} D^n$ where $K$ is a CW complex of dimension
$\le n {-}1$, and $\alpha\:S^{n-1} \to K$ is the attaching map for the
top cell of $X$.  Since $K$ is unique up to homotopy, we call $K$ the
{\it spine} of $X$.  Hodgson~\cite{Hodgson3} posed the question:

{\flushleft \bf Question 1 (Hodgson).}  {\it Given an $n$-dimensional
Poincar\'e complex $X^n$ with spine $K$, is there an $(n+2)$-dimensional
Poincar\'e complex $Y^{n+2}$ with spine $\S K$?}

Note that by Poincar\'e duality, the obvious dimension of $Y$ is
$n+2$.  The answer is often no, e.g., the cofibers of the Hopf
invariant one maps $\Bbb C P^2 = S^2 \cup_\eta D^4$,
$\Bbb H P^2 = S^4 \cup_\nu D^8$ and $S^8 \cup_\sigma D^{16}$, whose
spines are $S^2$, $S^4$ and $S^8$ respectively.  These examples are
generalized by the class pointed out to us by Jim Davis:



\begin{ex} 
\label{jims-ex}
Let $X$ be a connected $4k$-dimensional Poincar\'e complex with odd
Euler characteristic. Let $K$ be the spine of $X$.  Then there is no
Poincar\'e complex $Y$ of dimension $4k+2$ having spine $\S K$,
because the Euler characteristic of $Y$ would have to be odd (since
$\chi(X) \equiv \chi(Y) \text{ mod } 2$). But the Euler characteristic
of $Y$ must be even, since its intersection form is skew
symmetric.  Thus $Y$ can't exist.

The Poincar\'e complexes $\Bbb CP^{2k}$ and $\Bbb H P^{2k}$ are in
this class.  The class is closed with respect to taking products.
Furthermore, if $X^{4k}$ is in the above class and $Y^{4k}$ has even
Euler characteristic, then the connected sum $X \# Y$ is in the class.
\end{ex}

Question 1 sometimes has a positive answer: for example, the torus
$S^p {\times} S^q$ has spine
$S^p \vee S^q$. 
We formulate a slightly weaker version of Hodgson's question.

{\flushleft \bf Question 2.}  {\it Given a Poincar\'e complex $X$ with
spine $K$, does there exist an integer $j > 0$ and a Poincar\'e complex
$Y$ whose spine is $\Sigma^j K$?}

Adams's Hopf invariant one theorem ~\cite{Hopf-invariant-one} and 
$X = S^8 \cup_{\sigma} D^{16}$
shows that Question 2 can have a negative answer.  Question 2 has an
affirmative answer for those $X$ whose top cell splits off after a
single suspension, and for $j$ even:

\begin{bigthm} \label{period-thm} 
Let $X^n$ be a Poincar\'e complex with spine $K$ such that the top cell
of $X$ splits off after one suspension.  Then there exists a
Poincar\'e complex $Y^{n+4}$ whose spine is $\Sigma^2 K$ and whose top
cell splits off after one suspension.
\end{bigthm}


If a Poincar\'e complex embeds in codimension one, then its top cells
split off after one suspension: if
$X^n \subset S^{n+1}$ is a codimension one Poincar\'e embedding,
then the Pontryagin-Thom construction gives a degree one map
$S^{n+1} \to \S X$ which splits off the top cell of $X$.  We
answer Question 1 for this class of Poincar\'e complexes:

\begin{bigthm} \label{cod-one} 
Suppose $X$ as above has spine $K$. Then $\S K$ is the spine of a
Poincar\'e complex $Y$, and $Y$ has a codimension one
Poincar\'e embedding in $S^{n+3}$.
\end{bigthm}

The hypothesis of Theorem ~\ref{cod-one} implies that of Theorem
~\ref{period-thm}. We show that the converse need not hold: we will
construct infinite families of Poincar\'e complexes whose top cell falls
off after one suspension, but which fail to Poincar\'e embed in a
sphere in codimension one.  
Note however by a result of
Browder~\cite{Browder} that any such example must necessarily
embed in codimension two.
 
See \Sec2 for the definition of Whitehead products, and let 
$x,y\: S^n \to S^n \vee S^n$ be the inclusions into each summand.  Our
first infinite family is given by the ``Kervaire'' PL manifolds
(cf.~\cite[p.~120, Cor. 4.7]{Kosinski}).

\begin{bigprop} 
\label{Kervaire-example}
For any odd whole number $n \ne 1,3,7$, the cofiber  of the map
$$
[x, x] + [y, x] + [y, y]\: S^{2n-1} \to S^n \vee S^n
$$
is a $2n$-dimensional Poincar\'e complex $A_n$ which does not embed
in codimension one but whose top cell falls off after one suspension.
\end{bigprop}

We assume the reader is familiar with Toda's book~\cite{Toda}.  Recall
that at the prime 2 one has an EHP-sequence with connecting map
$P\: \pi_{*+2} (S^{2n+1}) \to \pi_*(S^n)$.  Then following result
provides criteria for constructing infinite families of examples.

\begin{bigprop} 
\label{spherical-fibration}
Given a map $\alpha\: S^{p+q+1} \to S^{2q+1}$ with order $2^{r+1}$,
with $q$ even, $p \ge 2q$ and $r > 1$, let $A$ be the cofiber of the
map
$$
[y, x] + y P(\alpha)\: S^{p+q-1} \to S^p \vee S^q .
$$
If $2^r$ kills the image of $E\:\pi_p (S^q) \to \pi_{p+1}(S^{q+1})$,
then $A$ is a $(p+q)$-dimensional Poincar\'e complex whose top cell
falls off after one suspension. However, $A$ does not embed in
codimension one.
\end{bigprop}


The Poincar\'e complex $A^{19}$ with attaching map
$
[y, x] + y P (E^5\sigma)\: S^{18} \to S^{13} \vee S^6
$ is an example of Proposition~\ref{spherical-fibration}, as Toda's
first table~\cite[p.~186]{Toda} shows the image of
$E\:\pi_{13}(S^6)\to \pi_{14}(S^7)$ has order at most $4$, whereas
$E^5 \sigma \in \pi_{20}(S^{13})$ has order $16$.  

Recall the Adams self-map
$W\: M^{n+8}_{16 \iota} \to M^n_{16\iota}$ of the Moore space which exists for $n > 10$
for stability reasons~\cite[Lem.~12.5]{Adams}.  We also need an
unstable Adams self-map:

\begin{lem}
\label{Adams-self-map}
For $n\ge 9$, there exists a map
$V\: M^{n+8}_{8 \iota}\to M^n_{8 \iota}$ so that the composite 
$
S^{n+7} @> i >>  M^{n+8}_{8 \iota} @> V >> M^n_{8 \iota} @>j >>  S^n
$ is homotopic to $2 \sigma$.
\end{lem}

We will use Adams $e$-invariant~\cite{Adams} to give a simple proof of
the following result, known to Mahowald
\cite[Thm.~1.5]{Mahowald:Image-J-EHP}, and possibly also known to
Barratt and Toda.

\begin{bigthm} \label{barratt-mahowald}
\begin{enumerate}
\item
There exist homotopy classes $N_k \in \pi_{8k}(S^5)$ of
order 8, where $N_1 = \nu$, and for $k > 1$, $N_k$ is the composite
$$
S^{8k} @> i >> M^{8k+1}_{8 \iota} @> V^{\circ{(k-1)}} >> M^9_{8 \iota}
@>\nu^\sharp >> S^5,
$$
where $\nu^\sharp$ is given by a nullhomotopy of 
$8 \nu\: S^8 \to S^5$.
\smallskip
\item
There exist homotopy classes $S_k \in \pi_{8k}(S^9)$ of order 16,
where $S_2 = \sigma$, and for $k > 2$, $S_k$ is the composite
$$
S^{8k} @> i >> M^{8k+1}_{16 \iota} @> W^{\circ (k-2)} >> 
M^{17}_{16 \iota} @> \sigma^\sharp >> S^9,
$$
where $\sigma^\sharp$ is given by a nullhomotopy of $16\sigma$.
\end{enumerate}
\end{bigthm}

Note that Mahowald has a powerful framework that explains and extends
these elements ~\cite[Thm.~1.5]{Mahowald:Image-J-EHP}.  Using these
two families and Proposition ~\ref{spherical-fibration}, we obtain

\begin{bigcor} \label{families}
The cofibers of the maps
\begin{align*}
[y, x] + y P(S_k)\: S^{8k-2} &\to S^{8k-5} \vee S^4
\\
[y, x] + y P(N_k)\: S^{8k-2} &\to S^{8k-3} \vee S^2
\end{align*}
are $(8k-1)$-dimensional Poincar\'e complexes whose top cells fall off
after one suspension, but do not embed in codimension one.
\end{bigcor}

The surgery exact sequence shows the above examples have the homotopy
type of smooth manifolds.  Other Poincar\'e complexes whose top
cell falls off after one suspension are provided by closing up Seifert
surfaces of high dimensional knots. However, we lack criteria for
deciding when these fail to embed in codimension one.

In \Sec\ref{Board-Steer} we review Boardman and Steer's work on Hopf
invariants and prove Theorem ~\ref{period-thm}.  We give a criterion
for Poincar\'e duality when the top cell splits off after one
suspension.  In \Sec\ref{cod-one-case} we prove Theorem
~\ref{cod-one}, and in \Sec\ref{spherical-fibration-proof},
Propositions ~\ref{spherical-fibration} and ~\ref{Kervaire-example}.
In \Sec\ref{barratt-mahowald-proof} we discuss Toda brackets and prove
Theorem ~\ref{barratt-mahowald} and Corollary ~\ref{families}.  We
explain in \Sec\ref{knot-period} how a variant of Theorem
~\ref{period-thm} gives rise to a periodicity operator for knot
theory, inducing the four-fold periodicity of the knot cobordism
groups.  In \Sec\ref{period-notes}, we discuss our linear notion of
periodicity, and explain some exponential periodicity of Mahowald.

\begin{acks} 
The authors are grateful to Mark Mahowald for his insight and
guidance.  We are also indebted Brayton Gray and Bob Bruner for
helping us understand how Adams's work ~\cite{Adams} proves
Theorem~\ref{barratt-mahowald}. The first author wishes to thank
Andrew Ranicki for help with the surgery theory literature, Diarmuid
Crowley for discussions in connection with Theorem
~\ref{spherical-fibration}, and Matthias Kreck and Peter Teichner for
pointing out that
$\S \Bbb RP^2$ is the spine of $\text{SU}(3)/\text{SO(3)}$.
\end{acks}

\section{Poincar\'e duality and Hopf invariants}
\label{Board-Steer}

The spaces in this paper are assumed to have the homotopy type of CW
complexes. Basepoints are always assumed to be non-degenerate.  If $X$
is a based space then $\S X$ denotes its reduced suspension, and
$\Omega X$ denotes its based loop space.  The smash product of based
spaces $A$ and $B$ is denoted $A\smsh B$.  Let $[A,B]$ denote the
(based) homotopy classes of maps from $A$ to $B$, and let $\{A,B\}$ be
the abelian group of stable homotopy classes of maps from $A$ to $B$.

See~\cite{Wall3} for the definition of a {\it Poincar\'e complex}.  We
consider only finite oriented Poincar\'e complexes.  If $X$ is an
$n$-dimensional Poincar\'e complex, there is a {\em fundamental class}
$[X] \in H_n(X)$ giving a cap product isomomorphism
\begin{equation}
\label{cap-product-isomorphism}
\cap [X] \: H^k(X) @>\cong>> H_{n{-}k}(X) \qquad\text{for all
  integers $k$.}
\end{equation} 
If $X$ is a 1-connected finite complex with a class $[X] \in H_n(X)$
satisfying~\eqref{cap-product-isomorphism}, then $X$ is a Poincar\'e
complex. Similar remarks hold for Poincar\'e pairs.

If $A$ and $B$ have the homotopy type of finite complexes, then a
(stable) map
$d\:S^n \to A\smsh B$ is an {\it $S$-duality map} if 
and only if the slant product homomorphism
$
/d_*[S^n] \: \tilde H^*(A) \to \tilde H_{n{-}*}(B)
$
is an isomorphism in all degrees. Here $\tilde H_*$ means reduced singular
homology, and $[S^n] \in \tilde H_n(S^n)$ denotes the generator.


We rely on Boardman and Steer's treatment of Hopf invariants
is~\cite{B-S}, much as we did earlier~\cite{Richter:Williams(1992)}.
Let $B$ be a based space.  The {\em suspension} map
$
E\: B\to \Omega \S B
$
is adjoint to the identity.
The {\it James Hopf invariant} 
$
H\:\Omega \S B \to \Omega\S (B \smsh B)
$ is a natural map~\cite[3.10]{B-S} with $H\circ E$ canonically null
homotopic.  (We will not need this, but
$
B @> E>> \Omega \S B @>H>>\Omega\S (B \smsh B)
$ is a metastable homotopy fiber sequence.)  $H$ gives a
natural map
$
H\:[\S A,\S B] \to [\S A,\S B \smsh B]
$.  The {\em Hopf invariant}
$
\lambda\:[\S A,\S B] \to [\Sigma^2 A,\S B \smsh \S B]
$ is the natural map~\cite[3.15]{B-S} given by suspending $H$.  
Boardman and Steer stress the Cartan formula ~\cite[thm.~3.15, def.~2.1]{B-S}: 
\begin{equation}
\label{Cartan-formula}
\lambda (f + g ) = \lambda (f) + f\cdot g + \lambda (g) \in [\S^2 A,
  \S B \^ \S B ],
\qquad\text{for $f, g\:\S A @>>> \S B$},
\end{equation}
where the {\em cup product} term $f\cdot g$ means the composite
$$
\S^2 A @>\S^2\Delta_A>> \S^2 (A \^ A) @>\text{shuffle}>> \S A\^\S A 
@>f\^ g >> \S B \^ \S B     .
$$
Following ~\cite{B-S}, we use right suspensions, so 
$\S A := A \^ S^1$, and suppress shuffle maps.   Note that by shuffling
the suspension coordinates around we can show that 
\begin{equation}
\label{cup-Whitehead-product-0}
f\cdot g = 0 \in [\S^2 A, \S B \^ \S B]\qquad
\text{if $\S f = 0 \in [\S^2 A, \S^2 B]$}
\end{equation}

Let $\tau_K \: K\smsh K \to K\smsh K$ be the twist map (which switches
factors).  The proof of ~\cite[Thm.~3.17]{B-S} (which assumed $B$ is a
suspension) generalizes to prove
\begin{equation}
\label{symmetrize-Hopf-invariant}
(f \^ f) \circ \S^2\Delta_A = (1 - \tau_{\S B}) \circ \lambda(f) +
\S^2\Delta_B \circ \S f \in [\S^2 A, \S B \^ \S B].
\end{equation}
As Boardman and Steer stress, the following diagrams commute up to
homotopy, for map $f, g\:\S A @>>> \S B$, because the twist
$\tau_{S^1}$ on
$S^1 \^ S^1$ has degree -1.  
\begin{equation}
\label{twist-and-suspend}
\begin{CD}
\S^2 B \^ B @>\text{shuffle}>> \S B \^ \S B
\\
@V -\S^2 \tau_B VV @V \tau_{\S B} VV 
\\
 \S^2 B \^ B @>\text{shuffle}>> \S B \^ \S B 
\end{CD}
\qquad\qquad
\begin{diagram}
\node{\S^2 A}   \arrow{e,t}{f\cdot g}   
\arrow{se,b}{-g\cdot f}
\node{\S B \^ \S B}   \arrow{s,r}{\tau_{\S B}}
\\
\node[2]{\S B \^ \S B}
\end{diagram}
\end{equation}

For maps $f\: \S P \to X$ and $g\: \S Q \to X$, the Whitehead product
~\cite[4.2]{B-S} $[f,g]\: \S P \^ Q \to X$ is defined as the unique
homotopy class so that the composite 
$\S (P \times Q) @>\S\pi_{12}>> \S P \^ Q @>[f,g]>> X$ is the commutator 
$(f\circ \S \pi_1, g\circ \S \pi_2)$.  We extend ~\cite[Thm.~4.6]{B-S}
to the case when $P$ and $Q$ are not required to be suspensions:

\begin{lem}
\label{Hopf-invariant-Whitehead-product}
Given maps $f\: \S P \to \S B$ and $g\: \S Q \to \S B$, the Whitehead
product
$[f,g] \: \S P \^ Q \to \S X$ has Hopf invariant
\begin{equation}
\label{Hopf-invariant-Whitehead-product:1+tau}
\lambda([f,g]) = (1 + \tau_{\S X}) \circ (f \^ g) \: \S^2 P \^ Q @>>> 
\S X \^ \S X.
\end{equation}
In particular, for the Whitehead product 
$[\iota, \iota]\: \S X\^X \to \S X$, we have 
$$
\lambda([\iota, \iota]) = 1 + \tau_{\S X}\: \S^2 X\^X @>>> \S X \^ \S X.
$$
\end{lem}

\begin{proof}
The map $\S\pi_{12}\: \S (P \times Q) \to \S (P \^ Q)$ is a stable
surjection, so it suffices to prove
~\eqref{Hopf-invariant-Whitehead-product:1+tau} pulled back to
$\S (P \times Q)$.  Write
$f_1 = f\circ \S \pi_1, g = g\circ \S \pi_2\: \S P \times Q \to \S X$.
By definition, 
$
[f,g]\circ \S\pi_{12} = (f_1, g_2) \in [\S (P \times Q), 
\S X]$. Write
$(f_1, g_2) = F - G$, where 
$F = f_1 + g_2$ and 
$G = g_2 + f_1$.  Then $F = (f_1, g_2) + G$.  By
~\eqref{cup-Whitehead-product-0}, the cup product $(f_1, g_2) \cdot G$
is nullhomotopic, because $\S (f_1, g_2)$ is nullhomotopic.  Thus
$$
\lambda(f_1) + f_1 \cdot g_2 + \lambda(g_2) = 
\lambda((f_1, g_2)) + \lambda(g_2) +    g_2 \cdot f_1 +\lambda(f_1)
$$
by the Cartan formula~\eqref{Cartan-formula}.  By
~\eqref{twist-and-suspend} and $f_1 \cdot g_2$ being a suspension, we
have
$$
\lambda((f_1, g_2)) = f_1 \cdot g_2 - g_2 \cdot f_1 
= 
(1 + \tau_{\S X}) \circ f_1 \cdot g_2
=
(1 + \tau_{\S X}) \circ (f\^g) \circ \S^2\pi_{12}.
\qed
$$
\renewcommand{\qed}{}
\end{proof}

Given a map
$f\: \S A \to \S B$ with cofiber $X$, the  diagram
\begin{equation}
\label{diagonal=Hopf-invariant}
\begin{CD}
\S A @>f>> \S B @>>> X @>\partial>> \S^2 A 
\\
@.  @. @V\Delta VV @VV \lambda(\alpha) V
\\
@. @. X \^ X @<<< \S B \^ \S B
\end{CD}
\end{equation}
is homotopy commutative~\cite[Thm.~5.14]{B-S}.  This
immediately implies

\begin{prop}
\label{B-S-criterion}
Let $L$ be a connected finite complex of dimension $\le n-3$, and
$\alpha\:S^{n-1} \to \S L$ a based map with cofiber $X$. Then $X$ is a
Poincar\'e complex if and only if $\lambda(\alpha) \:S^n \to \S L
\smsh \S L$ is an $S$-duality map.
\end{prop}


Suppose the cofiber of the based map $\alpha\:S^{n{-}1} \to K$ is a
Poincar\'e complex $X$ whose top cell splits off after one suspension,
so we have a degree one map
$\rho\:S^{n+1} \to \S X$.  Then 
$\S i\vee\rho\: \S K \vee S^{n+1} \to \S X$ is a homotopy equivalence
and defines a map
$f\:\S X \to \S K$ so that the composite $f\circ \S i\: \S K \to \S K$
is homotopic to the identity, and the composite
$f \circ \rho \: S^{n+1} \to \S K$ is nullhomotopic.

\begin{prop} 
\label{duality-criterion}
If $X$ is a Poincar\'e complex, the composite is an $S$-duality map:
$$
S^{n+2} @> \S \rho >>  \Sigma^2 X @> \lambda(f) >> \S K \smsh \S K
@> 1 - \tau_{\S K} >> \S K \smsh \S K .
$$ 
\end{prop}

\begin{proof} 
Applying the symmetrization formula~\eqref{symmetrize-Hopf-invariant}
to $f\: \S X \to \S K$ gives
\begin{equation*}
\label{symmetrizer}
(f \^ f) \circ \S^2\Delta_X = (1 - \tau_{\S K}) \circ \lambda(f) +
\S^2\Delta_K \circ \S f 
\in [\S^2 X, \S K \^ \S K].
\end{equation*}
Since $f \circ \rho $ is nullhomotopic, right composition with 
$\S \rho$ gives 
\begin{equation*}
\label{symmetrizer2}
(f \^ f) \circ \S^2\Delta_X \circ \S \rho = 
(1 - \tau_{\S K}) \circ \lambda(f) \circ \S \rho 
\in [S^{n+2}, \S K \^ \S K].
\end{equation*}
Relative Poincar\'e duality is given by a map 
$\tilde{\Delta}\: X @>>> K \^ K$, so the composite
$\Sigma\tilde{\Delta}\circ \rho$ is an $S$-duality map.  But
$\Delta_X\: X \to X \^ X$ is homotopic to the composite
$\Delta_X\: X @>\tilde{\Delta}>> K \^ K @>i\^i>> X \^ X$.  Thus 
$(f\smsh f) \circ \Sigma^2\Delta_X$ is homotopic to
$\Sigma^2\tilde{\Delta}$.  Hence 
$(f\smsh f) \circ \Sigma^2\Delta_X \circ \S \rho$ is homotopic to 
$\Sigma^2\tilde{\Delta}\circ \S \rho$, which is an $S$-duality map.
\end{proof}


\begin{proof}[Proof of Theorem A]
As above, let $X^n = K \cup_{\alpha} D^n$ be a Poincar\'e complex
whose top cell splits off after one suspension by a degree one map
$\rho\: S^{n+1} \to \S X$.  Suspend twice the $S$-duality map of
Proposition~\ref{duality-criterion}.  By using
~\eqref{twist-and-suspend}, we see the composition
\begin{equation}
\label{wrong-symmetry}
S^{n+4} @> \Sigma^3 \rho >>  \Sigma^4 X @> \S^2\lambda(f) >> 
\S^2 K \smsh \S^2 K
@> 1 + \tau_{\S^2 K} >> \S^2 K \smsh \S^2 K
\end{equation}
is an $S$-duality map.  
Define $\beta\: S^{n+3} \to \S^2 K$ as the composition
$$
\beta\: S^{n+3} @> \S^2 \rho >>  \Sigma^3 X @> \S\lambda(f) >> 
\S(\S K \smsh \S K) @> [\iota, \iota]>> \S^2 K.
$$ 
Let $Y$ be the cofiber of $\beta$.
Lemma~\ref{Hopf-invariant-Whitehead-product} and naturality shows
$\lambda(\beta)$ is composition~\eqref{wrong-symmetry}. Therefore
$\lambda(\beta)$ an $S$-duality map.  By
Proposition~\ref{B-S-criterion}, $Y$ is an $(n+4)$-dimensional
Poincar\'e complex.  Clearly the top cell splits off $Y^{n+4}$ after
one suspension, because the suspension of a Whitehead product is
nullhomotopic.
\end{proof}

Note that in above proof we could have tried unsuccessfully to
construct an $(n+2)$-dimensional Poincar\'e complex as the cofiber $Z$
of the composition
$$
\gamma\: S^{n+1} @> \rho >>  \Sigma X @> H(f) >> 
\S(K \smsh K) @> [\iota, \iota]>> \S K.
$$ 
The composition formula~\cite[Thm.~3.16]{B-S} calculates
$\lambda(\gamma)$ to be the composition
$$
S^{n+2} @> \S\rho >>  \Sigma^2 X @> \lambda(f)>> 
\S K \^ \S K @> 1 + \tau_{\S K}>> \S K \^ \S K.
$$ 
Since we have $1+\tau_K$ instead of $1-\tau_K$, we don't know that
$\lambda(\beta)$ is an $S$-duality map, so we can't conclude that $Z$
is an $(n+2)$-dimensional Poincar\'e complex.

\section{Periodicity in the codimension one case \label{cod-one-case}}
In this section we will prove Theorem ~\ref{cod-one}, solving Question
1 for the class of Poincar\'e complexes having codimension one embeddings
in the sphere.
For the definition of Poincar\'e embedding, see e.g.~\cite{Klein}.
Let $X^n$ be a connected oriented $n$-dimensional Poincar\'e complex
which is Poincar\'e embedded in $S^{n+1}$. Theorem ~\ref{cod-one} is
a direct consequence of the following.

\begin{prop} 
If $L$ denotes the spine of $X$, then $\S L$ is the spine of a
Poincar\'e complex $Y$, and $Y$ has a codimension one Poincar\'e
embedding in $S^{n+3}$.
\end{prop}

\begin{proof} 
The proof will rely on the decompression construction of
\cite[sec.~2.3]{Klein}.  By Spanier-Whitehead duality, the complement
of $X \subset S^{n+1}$ has two components, call them $M$ and $W$. The
normal data define inclusions
$X \to M$ and $X\to W$ which form Poincar\'e pairs of dimension $n+1$.
Then we have a homotopy pushout diagram
\begin{equation}
\label{embed-M-in-n+1}
\begin{CD}
X @>>> W
\\
@VVV @VVV
\\
M @>>> S^{n+1}
\end{CD}
\end{equation}
which gives a Poincar\'e embedding of $M$ with complement $W$.

The fiberwise suspension $S_M X$ of $X$ over $M$ is the double mapping
cylinder
$
S_M X = M \times 0 \cup X \times [0,1] \cup M \times 1 ,
$
together with the evident map $S_M X \to M$ (\cite[p.~609]{Klein}). Note that
$M \times 0$ provides a section $M \to S_M X$, and the map
$X \to W$ induces a map $S_M X \to \Sigma W$ given by collapsing each
copy of $M$ to a point.  Then the homotopy pushout diagram
$$
\begin{CD}
S_M X @>>> \Sigma W
\\
@VVV @VVV
\\
M @>>> S^{n+2}
\end{CD}
$$
is a Poincar\'e embedding of $M$ in $S^{n+2}$ with complement 
$\Sigma W$.  This is the {\em decompression} of $M$ in $S^{n+2}$,
which is well understood if $M$ is a closed submanifold of $S^{n+1}$,
and $X$ is the sphere bundle of the normal bundle.  Reversing the
roles of $M$ and
$\Sigma W$, we decompress once more to get a Poincar\'e embedding
$$
\begin{CD}
S_{\Sigma W} S_M X @>>> \Sigma W
\\
@VVV @VVV
\\
\Sigma M @>>> S^{n+3}\, .
\end{CD} 
$$
Set $Y = S_{\Sigma W} S_M X$, and note that the maps $Y\to \Sigma M$ 
and $Y \to \Sigma W$ have sections. The sum of these gives a map 
$\Sigma M \vee \Sigma W \to Y$ which is seen to be $(n+1)$-connected 
by application of the Mayer-Vietoris sequence to the diagram. The relative
Hurewicz theorem shows that $Y$ is obtained from $\Sigma M \vee \Sigma W$
by attaching an $(n+2)$-cell. So $\Sigma M \vee \Sigma W = \Sigma (M \vee W)$ is the spine of  $Y$. 
\end{proof}

For a harder proof similar to the proof of Theorem ~\ref{period-thm},
define
$Y$ as the cofiber of the composite
$S^{n+1} @>{\mathcal D}>> \S M \^ W @>[y, x]>> \S M \vee \S W$ (using
the $S$-duality map ${\mathcal D}$ of Lemma~\ref{non-embed} below),
whose Hopf invariant is an $S$-duality map by
Lemma~\ref{Hopf-invariant-Whitehead-product}. $Y$ is a Poincar\'e
complex by Proposition~\ref{B-S-criterion}.  There are obvious maps $Y
\to \S M$ and $Y \to \S W$, which one can show determines a Poincar\'e
embedding in $S^{n+2}$.

\section{\label{spherical-fibration-proof} Proof of Propositions ~\ref{spherical-fibration} and ~\ref{Kervaire-example}}


\begin{lem} 
\label{non-embed}  
Let $A = (S^p \vee S^q) \cup_\alpha D^{p+q}$ be a three cell complex
satisfying Poincar\'e duality, where $p,q > 1$.  If $A$ has a
Poincar\'e embedding in $S^{p+q+1}$, then there is a homotopy
equivalence
$A\simeq S^p \times S^q$.
\end{lem}

\begin{proof} 
Given a codimension one Poincar\'e embedding of $A$, its complement
has two components $M$ and $W$. We have a homotopy pushout
and a stable splitting ~\cite{Richter:Williams(1992)}
$$
\begin{CD}
A @>g >> W
\\
@Vf VV @VVV 
\\
M  @>>> S^{p+q+1}
\end{CD}
\qquad\qquad
\S A @>\S f + \S g + \S h>\simeq> 
\S M \vee \S W \vee S^{p+q+1},
$$
where $h\: A \to S^{p+q}$ is the pinch onto the top cell.  By the van
Kampen theorem, $M$ and $W$ are $1$-connected.  We will show that $M$
and $W$ are homotopy equivalent to the spheres $S^p$ and $S^q$, and
that the map $F = f \times g\: A \to M \times W$ is a homotopy
equivalence.  Let $A_0 = S^p \vee S^q$.  Since $A_0$ is a co-H space,
the restriction of $F$ to $A_0$ factors up to homotopy through the
wedge by a map $F_0 = x f + y g\: A_0 \to M \vee W$.  By the stable
splitting, $F_0$ is a homotopy equivalence by the Whitehead theorem.  

Assume $p \ne q$.  Since
$\Bbb Z \cong H_p(A_0) @>F_\ast>\cong> H_p(M) \oplus H_p(W)$, one of
the summands is 0.  Assume $H_p(M) \cong \Bbb Z$ and
$H_p(W) = 0$.  By Alexander duality, $H_q(M) = 0$ and 
$H_q(W) \cong \Bbb Z$.  Since $F_*$ is an isomorphism, all other
reduced homology groups of $M$ and $W$ vanish.  Thus we have homotopy
equivalences $M \simeq S^p$ and $W \simeq S^q$.

Assume $p = q$.  Now 
$\Bbb Z \oplus \Bbb Z \cong H_p(A_0) \cong H_p(M) \oplus H_p(W)$, and
all other reduced homology groups of $M$ and $W$ vanish.   By Alexander
duality, neither $M$ nor $W$ is contractible, so we have 
$\Bbb Z \cong H_p(M) \cong H_p(W)$, and again 
$M \simeq S^p$ and $W \simeq S^q$.

Thus we have shown that $F$ is a homology isomorphism except in degree
$p+q$, where $F$ is degree one because (cf.~\cite[Prop.~2.3]{Klein2},
~\cite[\Sec2]{Richter:Williams(1992)}) Alexander duality is induced by
the $S$-duality map
$$
{\mathcal D}\:
S^{p+q+1} @>\text{collapse}>> \S A @>\S \Delta>> \S A \smsh A 
@>>> \S M \smsh W \simeq S^{p+q+1} .
$$
Therefore, composing $F$ with the homotopy equivalences $M \simeq S^p$
and
$W \simeq S^q$ gives a homotopy equivalence
$A \to S^q \times S^p$.
\end{proof}

\begin{proof}[Proof of Proposition \protect{\ref{spherical-fibration}}]
$A$ is a Poincar\'e complex of dimension $p+q$ since the cup product
structure on $A$ is determined by the term $[y,x]$ appearing in the
attaching map, i.e., the cohomology ring of $A$ is just the cohomology
ring of $S^p \times S^q$.

By Lemma ~\ref{non-embed}, it suffices to show there is no map
$A \to S^q$ which has degree one on the $q$-cell.  Assume that such a
map exists.  By the  cofibration sequence
$$
S^{p+q-1} @>[y, x] + y P(\alpha)>> S^p \vee S^q @>>> A,
$$
there must be a map $f\: S^p \to S^q$ so the following composite is
nullhomotopic:
\begin{equation}
\label{nullhomotopic-if-embed}
S^{p+q-1} @>[y, x] + y P(\alpha)>> S^p \vee S^q @>f \vee 1>> S^q.
\end{equation}
This composite is 
$[\iota, f] + P(\alpha) \in \pi_{p+q-1} (S^q)$, by naturality of
Whitehead products.  But $[\iota, f]$ is homotopic to
the composite 
$S^{p+q-1} @>\Sigma^{q-1} f>> S^{2q-1} @>[\iota,\iota]>> S^q$, by the
Barcus-Barratt theorem~\cite{Barcus-Barratt(58)} and the
fact~\cite{Cohen:Seattle-course} that at the prime 2, all higher
Whitehead products vanish.  Now $[\iota,\iota]\: S^{2q-1} \to S^q$
is homotopic~\cite{Whitehead:book} to the composite
$P\circ E^2\: S^{2q-1} \to S^q$.  Thus 
$[\iota, f] = P (\Sigma^{q+1} f) \in \pi_{p+q-1}(S^q)$.  Since our
composite~\eqref{nullhomotopic-if-embed} is nullhomotopic, we have
$P (\Sigma^{q+1} f + \alpha) = 0 \in \pi_{p+q-1} (S^q)$.  By the
exactness of the EHP sequence,
$\Sigma^{q+1} f + \alpha = H(\beta) \in \pi_{p+q-1} (S^{2q+1})$ for
some $\beta\in \pi_{p+q-1} (S^{q+1})$.  By James's theorem
(cf.~\cite{Cohen:Seattle-course}), $2 H(\beta) = 0$, since $q$ is
even, so
$$
2 \alpha + 2 E^{q+1} f  = 0 \in \pi_{p+q+1} (S^{2q+1}).
$$
Multiplying this equation above by $2^{r-1}$ gives the contradiction
$2^r \alpha = 0$.
\end{proof}


\begin{proof}[Proof of Proposition \protect{\ref{Kervaire-example}}]
The top cell of $A := A_n$ falls off after one suspension since the
attaching map is a sum of Whitehead products.  $A$ is a
$2n$-dimensional Poincar\'e complex since the cup product structure on
$A$ is determined by the term $[y,x]$ appearing in the attaching map,
i.e.,
$H^*(A) \cong H^*(S^n \times S^n)$.

Assume $A$ has a Poincar\'e embedding in $S^{2n+1}$.  By Lemma
~\ref{non-embed}, there is a homotopy equivalence
$A \simeq S^n \times S^n$.  The projection of this equivalence onto
the first factor is a map $f\: A \to S^n$ such that $f^*(x)$ extends
to a basis of $H^n(A)$, where $x\in H^n(S^n)$ is a generator.  Thus
the restriction of $f$ to the $n$-skeleton $A_0$ is a map of the form
$a\iota \vee b\iota\:S^n \vee S^n \to S^n$, with $a$ and $b$
relatively prime.  By the cofibration sequence defining $A$, the
composite
$$
S^{2n-1} @>[x, x] + [y, x] + [y, y]>> S^n \vee S^n 
@>a\iota \vee  b\iota>> S^n 
$$
is nullhomotopic.  By naturality of Whitehead products,  this
composite is
$$
[a\iota, a\iota] + [b\iota, a\iota] + [b\iota, b\iota] = 
[\iota, \iota] a^2+ [\iota, \iota] ab + [\iota, \iota] b^2
= [\iota, \iota] (a+ab+b),
$$
using the left-distributivity of composition.  Since $n$ is odd, 
$2 [\iota,\iota] = 0$ (cf.~\cite{Cohen:Seattle-course}), but since 
$n \ne 1,3,7$, the Hopf invariant one
theorem~\cite{Hopf-invariant-one} implies that
$[\iota,\iota] \ne 0$.  Hence $[\iota,\iota]$ has order 2.
Thus
$
a + ab + b \equiv 0 \pmod 2
$, so 
$(a+1)(b+1) \equiv 1 \pmod 2$.  Thus both $a$ and $b$ are even.  But
$a$ and $b$ are relatively prime, so we have a contradiction.  Hence
$A$ does not embed in codimension one.
\end{proof}

\section{ Proof of Theorem ~\ref{barratt-mahowald} and Corollary ~\ref{families}}
\label{barratt-mahowald-proof}

The cofiber of a map $\gamma\: S^{u-1} \to S^v$ is called the {\em
Moore space} $M^u_\gamma$, and the cofiber $Y \cup_g CX$ of a map
$g\: X \to Y$ will also be written $M_g$.  The generator of
$\pi_7^s \cong \Bbb Z/16$ is 
$\sigma \in \pi_{15}(S^8)$, and~\cite[Prop.~5.15]{Toda} 
$2 \sigma \in \pi_7^s$ desuspends to the generator 
$\sigma' \in \pi_{14}(S^7) \cong \Bbb Z/8$.  Let
$\tau\: \S W \to Z$ belong to the Toda bracket
$\{h, g, f\}$ of the sequence
$W @>f>> X @>g>> Y @>h>> Z$.
$\tau$ is defined as a composite 
$SW @>f_\flat>> M_g @>h^\sharp>> Z$ defined using nullhomotopies of
$gf$ and $hg$.  The diagram 
\begin{equation}
\label{Adams-diagram}
\begin{CD}
M_f @>j>> \S W
\\
@V g^\sharp VV @VV \tau V 
\\
Y @>h>> Z
\end{CD}
\end{equation}
is homotopy commutative (cf.~\cite[Diag.~(5.1)]{Adams}), where
$g^\sharp\: M_f @>>> Y$ is defined using the same nullhomotopy of
$gf$.  Suppose the Toda bracket
$\{h, g, f\}$ contains $0$.  Then some
$\tau = h^\sharp\circ f_\flat$ is nullhomotopic.  By
diagram~\eqref{Adams-diagram}, the composite
$M_f @>g^\sharp>> Y @>h>> Z$ is nullhomotopic, and so defines a map
$V = (g^\sharp)_\flat\: \S M_f @>>> M_h$.  Thus the composite
$
\S X @>\S i>> \S M_f @>V>> M_h @>j>> \S Y
$ is homotopic to $\S g$.  
We now have


\begin{proof}[Proof of \protect{\ref{Adams-self-map}}]
Toda~\cite[Cor.~3.7]{Toda} shows the Toda bracket
$\{8\iota, E\sigma', 8\iota\} $ of the sequence
$
S^8 @< 8\iota << S^8 @< E\sigma' << S^{15} @< 8\iota << S^{15}
$
contains 0.  Thus we have a map
$M^{17}_{8 \iota}  @>V>>   M^9_{8 \iota}$
so that composite $j \circ V \circ i$ is homotopic to $E^2\sigma'$, which
is homotopic to $2 \sigma \in \pi_{16} (S^9)$.
The other maps with $n > 9$ are obtained by suspending $V$.
\end{proof}

\begin{proof}[Proof of Theorem \protect{\ref{barratt-mahowald}}]
Recall~\cite[Prop.~5.6]{Toda} that $\nu \in \pi_7(S^4)$ is the
generator of $\pi_8(S^5) \cong \pi_3^s \cong \Bbb Z/8$.
Adams~\cite[Lem.~12.5]{Adams} shows that $V$ is a $K$-theory
isomorphism, and
$e_{\Bbb C}(\nu) = a/4$ for some odd integer $a$, where
$e_{\Bbb C}$ the complex $e$-invariant (See~\cite[Prop.~7.14,
Ex.~7.17]{Adams})
The proof of ~\cite[Thm.~12.3]{Adams} shows that
$e_{\Bbb C}(N_k) = b/4$, where $b$ is odd.  Then 
$e_{\Bbb R}(N_k) = b/8$, where $e_{\Bbb R}$ is the real $e$-invariant defined
by suspending $N_k$ to have target sphere $S^8$.  Hence $N_k$ has
order at least 8.  The Adams self-map construction shows $8 N_k = 0$,
so $N_k$ has order 8.  

Adams~\cite[Ex.~7.17]{Adams} shows that $e_{\Bbb C}(\sigma) = r/16$
for some odd integer $r$.  Hence
$e_{\Bbb C}(S_k) = s/16$, for some odd integer $s$, and
$S_k$ has order at least 16.  The Adams self-map
construction shows $16 S_k = 0$, so $S_k$ has order 16. 
\end{proof}

\def\Z{{   \mathbb Z }}

We now give a longer proof of Theorem ~\ref{barratt-mahowald} which we
hope is more comprehensible.  
For the $S_k$ family, we require only that the Adams self-map is a
$K$-theory isomorphism, a fact that we believe was known to Barratt
prior to ~\cite{Adams}.

Recall the Toda bracket $\{h, g, f\}$ of a sequence
$W @>f>> X @>g>> Y @>h>> Z$.  Any map $\tau\: \S W \to Z$ making
diagram~\eqref{Adams-diagram} homotopy commute must belong to
$\{h, g, f\}$: we know some element of the Toda bracket
$\tau_0 \in \{h, g, f\}$ makes diagram~\eqref{Adams-diagram} homotopy
commute, and by exactness of the cofibration sequence
$M_f \to \S W @>\S f>> \S X$, we know that 
$\tau= \tau_0 + p\circ \S f \in [\S W, Z]$, for some element 
$p \in [\S X, Z]$.  But $p\circ \S f \in [\S W, Z]$ belongs to the
indeterminacy of $\{h, g, f\}$.  Hence $\tau \in \{h, g, f\}$.

By Bott periodicity, $\pi_{2n}(BU) \cong \Z$ and
$\pi_{2n-1}(BU) = 0$, for $n > 0$.  The generator 
$\zeta_n\: S^{2n} @>>> BU$ is the $n$-fold exterior power of the
bottom generator $\zeta_1 \in BU$, which comes of course from 
$S^2 = \Bbb CP^1 \subset \Bbb CP^\infty = BU(1) \subset BU$.  

Given $f \in \pi_{2m-1} S^{2k}$ with $q f = 0$, consider the
Toda bracket $\{\zeta_k, f, q \iota\} \subset \Z$ of the sequence
$ BU @<\zeta_k<< S^{2k} @<f<< S^{2n-1} @<q \iota<< S^{2n-1}
$.
The indeterminacy is $q \pi_{2n}(BU) = q\Z \subset \Z$, as $f$ has
finite order.   We often mod out by the indeterminacy and write 
$\{\zeta_k, f, q\iota\} \in \Z/q$.  We can use these Toda brackets to
establish the order of an element.  Suppose e.g.~that $16 f = 0$, and
we show that $\{\zeta_k, f, 16\iota\} \in \Z/16$ has order 16.  Then
we can show $f$ must have order 16.  By a Toda bracket identity, the
image of
$\{\zeta_k, f, 16\iota\} \in \Z/16$ under the projection $\Z/16 \to
\Z/2$ is $1 = \{\zeta_k, f 8\iota, 2\iota\} \in \Z/2$.  Hence 
$8 f = f 8\iota$ must be nonzero.

By the above, $r \in \{\zeta_k, f, q\}$ iff the diagram homotopy
commutes:
\begin{equation}
\label{delta-Moore-sharp}
\begin{CD}
M_{q\iota}^{2n} @>j>>  S^{2n}
\\
@V f^\sharp VV   @VV r \zeta_n V
\\
S^{2k} @>\zeta_k>> BU
\end{CD}
\end{equation}
Adams defined the complex Adams operation
$\Psi_2\: BU \to BU$, which defines a ring homomorphism in
$\tilde{KU}(X) = [X, BU]$ satisfying two properties: 
$\Psi_2(\zeta_n) = 2^n \zeta_n \in \pi_{2n}(BU)$, and 
$\Psi_2(x) = x \cup x \in KU(X) \pmod 2$, for any class
$x \in KU(X)$.    
Then $(\Psi_2 - 2^k) \zeta_k = 0$, and we have the Toda bracket
identity 
\begin{equation}
\label{delta-e-invariants}
(\Psi_2 - 2^k) \{\zeta_k, f, q\} =  \{(\Psi_2 - 2^k), \zeta_k, f\}  q
\in \pi_{2n} BU \cong \Z.
\end{equation}
The Toda bracket $\{(\Psi_2 - 2^k), \zeta_k, f\}$ is essentially
Adams's complex $e$-invariant, and it has indeterminacy $2^k$ times an
odd number ($2^{n-k} - 1$, in fact).  It easily follows from
~\cite{Adams}, or the properties above, that 
$\{(\Psi_2 - 2^k), \zeta_k, \nu\}$ is $2^{k-2}$ times an odd number,
for $k \ge 2$, and that$\{(\Psi_2 - 2^k), \zeta_k, \sigma\}$ is
$2^{k-4}$ times an odd number, for $k \ge 4$.  Then it follows from
~\eqref{delta-e-invariants} that 
 
\begin{lem}
\label{delta-invariant-nu-sigma}
The Toda bracket $\{\zeta_k, \sigma, 16\} \in \Z/16$ of the sequence
$$
BU @<\zeta_k<< S^{2k} @<\sigma<< S^{2k+7}  @<16<< S^{2k+7} 
\qquad\qquad\text{has order 16.}
$$

The Toda bracket $\{\zeta_k, \nu, 16\} \in \Z/8$ of the sequence
$$
BU @<\zeta_k<< S^{2k} @<\nu<< S^{2k+3}  @<8<< S^{2k+3} 
\qquad\qquad\text{has order 4.}
$$
\end{lem}

\begin{proof}
Choose $r \zeta_{k+4} \in \{\zeta_k, \sigma, 16\}$.  Then for $a$, $b$
odd,
$$
2^k a r \zeta_{k+4} = (\Psi_2 - 2^k) r \zeta_{k+4}
=
\{(\Psi_2 - 2^k), \zeta_k, \sigma\}  16
=
2^{k-4} b 16 \zeta_{k+4}
$$
modulo the indeterminacy $2^{k+4}$, so $r$ is odd.  $\nu$ is handled
similarly.
\end{proof}

Thus the composite (recall that $j\circ W = \sigma^\sharp$)
$$
M^{2k+8}_{16 \iota} @>W>> M^{2k}_{16\iota} @>j>> S^{2k} @>\zeta_k>> BU
$$
is homotopic to 
$M^{2k+8}_{16 \iota} @>j>> S^{2k+8} @>r \zeta_{k+4}>> BU$, where $r$
is odd.  This is what is meant by saying that $W$ is a (2-local)
$K$-theory isomorphism.  Now suspend $S_k$ 7 times:
$$
S^{8k+7} @> i >> M^{8k+8}_{16 \iota} @> W^{\circ (k-2)} >> 
M^{24}_{16 \iota} @> \sigma^\sharp >> S^{16}.
$$
By induction, the composite
$
M^{8k+8}_{16 \iota} @> W^{\circ (k-2)} >> 
M^{24}_{16 \iota} @> \sigma^\sharp >> S^{16} 
@>\zeta_8>> BU 
$
is homotopic to an odd multiple of the generator 
$M^{8k+8}_{16 \iota} @>j>> S^{8k+8} @>\zeta_{4k+k}>> BU$.  By
~\eqref{delta-Moore-sharp}, $\{\zeta_8, S_k, 16\iota\} \in \Z/16$ is
odd, and hence of order 16.  Thus
$\{\zeta_8, S_k 8\iota, 2\iota\} = 1\in \Z/2$, and we have proved that
$S_k$ has order 16.

The case of $N_k$ is similar, but harder.  By
Lemma~\ref{delta-invariant-nu-sigma} and ~\eqref{delta-Moore-sharp},
the composite
$M^{12}_{8 \iota} @>\nu^\sharp >> S^8 @>\zeta_4>> BU$ is homotopic to
twice an odd multiple of 
$M^{12}_{8 \iota} @>j>> S^{12} @>\zeta_6>> BU$.
Suspend $N_k$ 3 times:
$$
S^{8k+3} @> i >> M^{8k+4}_{8 \iota} @> V^{\circ{(k-1)}} >> 
M^{12}_{8 \iota} @>\nu^\sharp >> S^8.
$$
The composite 
$
M^{8k+4}_{8 \iota} @> V^{\circ{(k-1)}} >> 
M^{12}_{8 \iota} @>\nu^\sharp >> S^8 @>\zeta_4>> BU$ is homotopic to
twice an odd multiple of 
$
M^{8k+4}_{8 \iota} @>j>> S^{8k+4} @>\zeta_{4k+2}>> BU
$
and we conclude that $N_k$ has order at least 4, which is good enough
for our Poincar\'e embedding results.


\subsection*{Proof of Corollary ~\ref{families}} 
For the Poincar\'e complexes defined using $N_k$, we consider the
image of suspension homomorphism $E\:\pi_p(S^2) \to
\pi_{p+1}(S^3)$. The composite
$$
\pi_p(S^3) @> \eta_\ast >\cong > \pi_p(S^2) @> E >> \pi_{p+1}(S^3)
$$ 
coincides with $x \mapsto E\eta \cdot Ex$ and $E\eta$ has order
two. Hence, the image of $E\:\pi_p(S^2) \to \pi_{p+1}(S^3)$ is killed
by 2.  By Proposition ~\ref{spherical-fibration}, this gives the
result, since $N_k$ has order $8$. (Alternatively, we could have
used~\cite{Selick}, since $S^3$ has exponent 4.)

In the case of the Poincar\'e complexes defined using $S_k$, we need
to consider the image $E\: \pi_p(S^4) \to \pi_{p+1} (S^5)$.  By
Selick's theorem, $S^5$ has exponent $8$, and therefore the image of
$E$ is killed by $8$.  Since $S_k$ has order 16, and the conclusion
follows once again by application of Proposition
~\ref{spherical-fibration}.

\section{Periodicity in high dimensional knot theory \label{knot-period}}

We show how Theorem ~\ref{period-thm} gives a homotopy-theoretic
periodicity operator from $n$-knots to $(n+4)$-knots, inducing the
four-fold periodicity in the knot cobordism groups
~\cite{Levine:cobordism}.  Knot periodicity has been geometrically
described ~\cite{Bredon:periodicity, C-S:cobordism, Kauffman}.

Fix $n \ge 1$. By a {\it (smooth) Seifert surface} we mean 
an codimension one compact smooth submanifold $V^{n+1}\subset S^{n+2}$
in which $\partial V := \Sigma^n$ is a homotopy $n$-sphere. 

Two Seifert surfaces $V_i \subset S^{n+2}$
with $i = 1,2$ are said to be {\it equivalent} if
there is a diffeomorphism of $S^{n+2}$ which transfers $V_1$ to $V_2$.

\begin{rem} 
If $\Sigma^n \subset S^{n+2}$ is a codimension two knot, then
it has a Seifert surface. If the fundamental group of the complement
of the knot is infinite cyclic, then there exists a Seifert surface for
it which is simply connected ~\cite{Levine_unknot}.  Conversely, if
there is a $1$-connected Seifert surface, then the complement has
infinite cyclic fundamental group. One says in this instance that the
knot is {\it $1$-simple}.
\end{rem}

\subsection*{Homotopy Seifert Surfaces} Fix $n \ge 2$. A {\it homotopy Seifert
surface} of dimension $n+1$ is a diagram of spaces
$
\xymatrix{
S^n \ar[r]^{\alpha} & K \ar@{=>}[r]_{p_-}^{p_+} & C
}
$
in which
\begin{itemize}
\item $\alpha$ is an inclusion making $(K,S^n)$ into a Poincar\'e pair.
\item $K$ and $C$ are $1$-connected and have the homotopy type of finite CW complexes;
\item $p_- \circ \alpha = p_-\circ \alpha$;
\item The homomorphism $(p_+)_* - (p_+)_*\: H_*(K) \to H_*(C)$ is an
isomorphism  in positive degrees, where $H_*$ denotes singular homology.
\end{itemize} 
(Compare ~\cite{Richter,Farber}.) Denote these data by
$(\alpha,p_\pm)$.
An equivalence $(\alpha,p_\pm) @>\sim >> (\alpha',p'_\pm)$ (with 
$p'_\pm \: K' \to C'$) consists homotopy equivalences $a\: K \to K'$
and $b\: C \to C'$ such that $q_\pm \circ a = b\circ p_\pm$ and
$\alpha' = a \circ \alpha$. Two homotopy Seifert surfaces will be
called {\it equivalent} if there is a finite chain of equivalences
connecting them.

\begin{lem} \label{null-lemma} If $(\alpha,p_\pm)$ is a homotopy Seifert Surface, then 
 $\Sigma\alpha$ is nullhomotopic. Furthermore, the homotopy class of the
 nullhomotopy is preferred.
\end{lem}

\begin{proof} The map $\S p_+ - \S p_-\: \S K \to \S C$ is
a homology isomorphism and therefore a homotopy equivalence by the
Whitehead theorem.  Call this map $h$. Then $h \circ \S \alpha$
has a preferred nullhomotopy.  The nullhomotopy for $\S \alpha$ is
now obtained by choosing a homotopy inverse for $h$.
\end{proof}

\subsection*{The relation between smooth Seifert  and homotopy Surfaces}
Let $V^{n+1} \subset S^{n+2}$ be a simply connected Seifert surface with $\partial V = \Sigma$.
We will show how to construct an associated homotopy Seifert surface.

Fix an orientation preserving   $\S \cong S^n$ homeomorphism (here we
are using the Poincar\'e conjecture).
Choose a compact tubular neighborhood $U$ of $V$ and define $C$ to be the complement of
the interior of $U$. Then $\partial U \subset C$.  Identify $U$ with $V \times I$. Then
$\partial U$ is identified with $V\times 0 \cup \Sigma^n \times I \cup V\times 1$.

Let $K_- := V \times 0 \cup \Sigma^n \times [0,1/2]$ and $K_+ = V \times 1 \cup \Sigma^n \times [1/2,1]$. Then $K_-$ and $K_+$ are homeomorphic by a preferred homeomorphism $h\: K_- \to K_+$. 
Set $K := K_-$ and let $\alpha\: S^n \to K$ be the identification $S^n \simeq \S \times 1/2$
followed by the inclusion $\S \times 1/2 \subset K_-$. Define $p_-\: K \to C$ to be
the inclusion, and $p_+\:K \to C$ to be $h\: K = K_- \to K_+$ followed by the 
inclusion $K_+ \subset C$.  By construction $p_\pm$ coequalize $\alpha$ and 
$(K,S^n)$ is a Poincar\'e pair. The homomorphism $(p_+)_* - (p_-)_*$ is seen to be an isomorphism in
positive degrees using the pushout diagram
$$
\begin{CD}
\partial (V\times I) @>>> C
\\
@VVV @VVV 
\\
V \times I  @>>> S^{n+2}
\end{CD}
$$
as follows: let $D_0$ be the result of removing the top cell of $\partial (V\times I)$.
Then $D_0$ is identified with $K \vee K$ up to homotopy.  With respect
to this identification we have a homotopy pushout
$$
\begin{CD}
K \vee K @>>> C
\\
@VVV @VVV 
\\
K  @>>> \Bbb R^{n+2}
\end{CD}
$$
where $K \vee K \to K$ is the fold map  and 
$K \vee K \to C$ is the map $(p_-,p_+)$. The conclusion now follows from the Mayer-Vietoris
sequence for the pushout.

\begin{thm} \label{smooth} Assume $n \ge 5$. Then the above induces a
bijection between the set of equivalence classes of $1$-connected smooth Seifert surfaces
in $S^{n+2}$ and the set of equivalence classes of homotopy Seifert surfaces of dimension
$n+1$.
\end{thm}

\begin{proof} (Existence). We need to show that the every homotopy Seifert 
surface arises up to equivalence
from a smooth one. Let $(\alpha,p_\pm)$ be a homotopy Seifert surface, with $\alpha\: S^n \to K$
and $p_\pm \: K \to C$. Let $D(K)$ denote the double mapping cylinder 
$K\times 0 \cup S^n \times I \cup  K\times 1$, and  let
$p\: D(K)\to C$ be the map defined by $p_-$ on $K\times 0$, $p_+$ on $K \times 1$ and
the constant homotopy of the map $p_-\circ \alpha$ on $S^n \times I$. Without loss
in generality, we can assume $p\:D(K) \to C$ is a cofibration. Let 
$$
N = (K \times I) \cup_{D(K)} C\, .
$$
Then $N$ has an orientation preserving homotopy equivalence to $S^{n+2}$. Furthermore,
we have a Poincar\'e triad $(N;K\times I,C;D(K))$.

The scheme will be to use the diagram of smooth structure sets
$$
\multiply\dgARROWLENGTH by 2
\divide\dgARROWLENGTH by 3
\begin{diagram}
\node[2]{\cal S(N;K\times I,C;D(K))}   \arrow{e,tb}{\Phi_1}{\cong}   
\arrow{s,r}{\Phi_2}
\node{ \cal S(N) \cong \cal S(S^{n+2})}   
\\
\node{\cal S(K,S^n)}   \arrow{e,tb}{\cong}{\times I}   
\node{\cal S(K\times I,D(K)).}   
\end{diagram}
$$
The $h$ or $s$ decorations on the structure sets are unnecessary since we are in the 
simply connected case. Here, for an $n$-dimensional Poincar\'e pair $(X,\partial X)$ of $1$-connected complexes, 
$\cal S(X,\partial X)$ denotes the set generated by homotopy equivalences of pairs 
$(M,\partial M) \to (X,\partial X)$ subject to the relation of $h$-cobordism. Similarly,
$\cal S(N;K\times I,C;D(K))$ is the smooth structure set on the Poincar\'e triad
$(N;K\times I,C;D(K))$. The function labeled $\Phi_i$ are forgetful
maps, and $\Phi_1$ is an isomorphism by codimension one splitting
~\cite[Thm.~12.1]{Wall}. The function labeled ``$\times I$'' is given
by taking cartesian product with the unit interval.  It too is an
isomorphism by the $\pi$-$\pi$ theorem ~\cite[Thm.~3.3]{Wall}.

We proceed as follows. Choose the identity structure on $S^{n+2}$ and use
the top isomorphism of the displayed diagram to give a smooth triad structure 
$$
(S^{n+2};U,C';\partial U) @>{\sim}>> (N;K\times I,C;D(K))
$$
Then use the bottom isomorphism of the diagram to write $(U,\partial U)$ as $(V\times I,\partial (V\times I))$
up to diffeomorphism where $\S := \partial V$ is a homotopy $n$-sphere.  Then we have
a smooth triad $(S^{n+2};V\times I,C',\partial (V\times I))$ yielding a smooth $1$-connected
Seifert surface $V\times 1/2\subset S^{n+2}$.  It is clear that the homotopy Seifert surface associated with
the smooth one is equivalent to the one we started with.

\noindent (Uniqueness).  The proof will also appeal to the diagram appearing the proof
of existence. Let $(\alpha, p_\pm)$ be a homotopy Seifert surface as above
and suppose that $V_i \subset S^{n+2}$ are $1$-connected Seifert surfaces, $i = 0,1$,
whose associated homotopy Seifert surfaces admit equivalences to $(\alpha, p_\pm)$.
The equivalences yield a pair of two smooth triad structures
$$
(S^{n+2};V_i\times I,C'_i,\partial (V_i\times I)) @> \sim >> 
(N;K\times I,C;D(K)) \, ,
$$
and by using the injectivity of $\Phi_1$, we infer
that the two smooth triad structures are equivalent. We infer 
(by straightening $h$-cobordisms) that there is a diffeomorphism
$$
\psi\: (S^{n+2};V_0\times I,C'_0,\partial (V_0\times I)) \cong  (S^{n+2};V_1\times I,C'_1,\partial (V_1\times I))\, .
$$
Using the injectivity of the function $\times I$, it follows
that the restricted diffeomorphism $\psi\: V_0 \times I \to V_1 \times I$ is 
pseudoisotopic to one of the form $\phi\times \text{id}$, where $\phi\: V_0 \to V_1$ is a diffeomorphism. Choose such a pseudoisotopy and  let 
$$
H\: \partial (V_0 \times I) \times [0,1] @> \cong >> \partial (V_1 \times I) \times [0,1]
$$ 
be its restriction to the boundary. 
Choose  collar neighborhoods $T_i \cong \partial (V_i \times I) \times [0,1]$
of   $\partial (V_i \times I) \subset V_i \times I$. Then $H$ defines a 
diffeomorphism $T_0 \cong T_1$ which extends to a diffeomorphism
$H'\:V_0 \times I \to V_1 \times I$ by taking $\phi \times \text{id}$ on the complement
of $T_0$. Extend $H'$ to a diffeomorphism of $S^{n+2}$ using
$\psi\: C'_0 \to C'_1$. The constructed diffeomorphism of $S^{n+2}$ takes $V_0 \times 1$
to $V_1 \times 1$, so we get an equivalence of between the smooth Seifert surfaces.
\end{proof}

\subsection*{Knot periodicity}
We define an operator which associates to a
 homotopy Seifert surface in dimension $n+1$
another one of dimension $n+5$. 

Let $(\alpha,p_\pm)$ be a homotopy Seifert Surface of dimension $n+1$, where
$\alpha\: S^n \to K$ and $p_\pm\: K \to C$.  Theorem ~\ref{period-thm}
produces an attaching map $\beta \: S^{n+4}\to \Sigma^2 K$ whose
cofiber satisfies Poincar\'e duality. The proof shows that
$\Sigma^2p_-\circ \beta$ and $\Sigma^2p_+\circ \beta$ are homotopic 
via a preferred homotopy $f\:S^n \times I \to C$ (the verification of this is straightforward but tedious; we therefore omit it). The maps $\beta\:S^{n+4}\to \Sigma^2K$ and
$\Sigma^2p_\pm\: \Sigma^2 K \to \Sigma^2 C$ are close to defining a homotopy Seifert surface.
However, there are two defects: (1) $p_\pm$ is only known to coequalize $\beta$ up to homotopy, and 
(2) $\beta$ is not a cofibration. We will show how to fix these problems.

Factor the map $\beta$ by a cofibration $\beta'\: S^n \to Z$ followed by a homotopy equivalence
$h\:Z \to \Sigma^2K$. Let $p'_\pm\: Z \to C$ be $\Sigma^2 p_\pm \circ h$. 
Then $f$ defines a homotopy from $p'_-\circ \beta'$ to $p'_+\circ \beta'$.  By the homotopy
extension property, we obtain a map $q_-\: Z \to C$ such that $q_-\circ \beta' = 
(\Sigma^2 p_+) \circ \beta'$. Set $q_+ = \Sigma^2 p_+$.  Then $(\beta',q_\pm)$ is a homotopy
Seifert surface.

We now sketch a proof that the assignment $(\alpha,p_\pm) \mapsto (\beta',q_\pm)$
yields four-fold periodicity in knot cobordism. Although the verification is
somewhat tedious, the basic idea is that the intersection pairing 
of $X = K \cup_\alpha D^{n+1}$ together with the homomorphism that $p_+$ induces
on homology completely determines the smooth knot cobordism class of $(\alpha,p_\pm)$
(here we are implicitly using ~\ref{smooth} to identify $(\alpha,p_\pm)$ with a smooth
Seifert surface to make sense of the smooth knot cobordism class 
of the homotopy Seifert surface).
Then the result is established
once we show that the intersection pairing of $Y := (\Sigma^2 K)\cup_\beta D^{n+5}$
has the same intersection pairing as $X$ up to regrading (since $q_\pm$ and $p_\pm$
induce the same homomorphisms on homology).
That is idea. Some details follow.

Note that the basepoint for $S^n$
gives basepoints for $K$ and $C$. 
The maps $D(K) \to K \times I \to K$
and $D(K) \to C$ combine to a give a map $D(K) \to K \times C$, which we follow up with
the quotient map $K \times C \to K\smsh C$ to obtain a map $D(K) \to K \smsh C$.
The commutative diagram
$$
\begin{CD}
C @<<<  D(K) @>>> K
\\
@VVV @VVV  @VVV 
\\
\ast @<<<  K\smsh C  @>>> \ast
\end{CD}
$$
induces a map of homotopy pushouts $d\: S^{n+2} \to \S K\smsh C$
which is an $S$-duality map, which in turn yields the Alexander
duality isomorphism $H_*(K) \cong H^{n+1-*}(C)$ in positive
degrees. Then the homology class $d_*([S^{n+1}]) \in H_{n+1}(K\smsh
C)$ determines a class $d^\sharp \in H^{n+1}(K\smsh C)$ which are
Alexander dual via the duality map $d\smsh d$. Let
$$
\delta\: H_j(K) \otimes H_{n+1-j}(C) \to  \Bbb Z
$$
be given by $\delta(a\otimes  b) = d^\sharp (a\times b)$. Then $\delta$ is 
the {\it Alexander pairing}.

\begin{defn} The {\it Seifert pairing} 
$$
\Phi\: H_j(K) \otimes H_{n+1-j}(K) \to  \Bbb Z
$$
of $(\alpha,p_\pm)$ is given by 
$\Phi(x\otimes y) = \delta((p_+)_*(x)\otimes y))$.
\end{defn}

To establish periodicity, it will be enough by Levine
~\cite{Levine:cobordism} to show that that the Seifert pairings for
$(\alpha,p_\pm)$ and $(\beta',q_\pm)$ coincide.  To keep the
discussion simple, we will only verify this when $K$ is a suspension
(this is sufficient because Levine showed that every smooth $n$-knot
is cobordant to one having a Seifert surface which is
$\lfloor(n+1)/2\rfloor$-connected ~\cite{Levine_unknot}, and such
Seifert surfaces desuspend by the Freudenthal theorem).  We may
therefore assume that $K$ is $\lfloor(n+1)/2\rfloor$-connected, $K$ is
a suspension $\S L$, and $C$ is also identified with $\S L$
using $p_+ - p_-\: \S L \to C$.  Then map $\alpha\: S^n \to \Sigma
L$ factors as
$$
S^n @> \hat\alpha >>  \S L \smsh L @> P >> \S L
$$
where $P$ is the Whitehead product.  Furthermore, the composite
$$
 S^{n+1} @> \S \hat\alpha >>  \S L \smsh \S L @> 1 +\tau_{\S L} >> \S L \smsh \S L
$$
is an $S$-duality map.
Likewise, the proof of Theorem ~\ref{period-thm}
shows that the attaching map $\beta\: S^{n+4} \to \Sigma^2 L$ is given by 
$$
S^{n+4} @> \Sigma^4 \hat\alpha >> \S (\Sigma^2 L) \smsh (\Sigma^2 L) 
@> P >> \S (\Sigma^2 L).
$$
It is clear from this description that the intersection pairings for 
$X^{n+1} = K \cup_\alpha D^{n+1}$ and 
$Y = \Sigma^2 K \cup_\beta D^{n+5}$ coincide after regrading, since
the cup product structure of $X$ is completely determined by the
homomorphism induced by 
$\phi\:= (1 +\tau_{\S L})\circ (\S \hat\alpha)$ on
homology. More precisely, by Boardman and Steer ~\cite{B-S}, there is
a homotopy commutative diagram
$$
\begin{CD}
X @>\Delta>> X\smsh X
\\
@VVV @AAA
\\
S^{n+1}   @>\phi>> \S L \smsh \S L,
\end{CD}
$$
where $\Delta$ is the diagonal (inducing the cup product), the left
vertical map is the pinch map onto the top cell, and the right
vertical map is the inclusion.

Notice the inclusion $K \to X$ induces an isomorphism in homology in 
degrees $\ne n+1$.  Furthermore there is a  map $\S X \to \S C$
which is a homology isomorphism in degrees $\ne n+1$. The latter map is defined as follows:
the Poincar\'e embedding gives an equivalence between
$\S X$ with the cofiber $S^{n+2} \cup \text{Cone}(C)$;
compose this equivalence with the connecting map $S^{n+2} \cup \text{Cone}(C) \to \S C$
appearing in Barratt-Puppe sequence. (In terms of the splitting, $\S X \simeq \S K \vee S^{n+2}$, the restriction of the map $\S X \to \S C$ to $\S K$ is
identified with the homotopy equivalence $\S p_+ - \S p_-$, whereas
the restriction to the $S^{n+2}$ summand is trivial.)

Thus the intersection pairing of $X$ can be rewritten in positive
degrees as
$$
H_j(X) \otimes H_{n+1-j}(X) \cong H_j(K) \otimes H_{n+1-j}(C) \to  \Bbb Z\, ,
$$
where the second homomorphism is the Alexander pairing $\delta$. Thus,
the intersection pairing of $X$ and the Alexander pairing of the
Poincar\'e embedding associated with $(\alpha,p_\pm)$ coincide in
positive degrees. A similar statement holds for $Y$.

Since the intersection pairings for $X$ and $Y$ coincide (after
regrading), the Alexander pairings arising from $(\alpha,p_\pm)$ and
$(\beta',q_\pm)$ also coincide.  Since $q_\pm$ coincides with $p_\pm$
on homology, the Seifert pairings of $(\alpha,p_\pm)$ and
$(\beta',q_\pm)$ coincide.

\section{The period of a finite complex}
\label{period-notes}

Theorem ~\ref{period-thm} is not the most general result.  If 
$X = \Bbb RP^3$, then the spine of $X$ is $\Bbb RP^2$, and the top
cell of $X$ splits off after two suspensions but not one.  On the
other hand,
$\Sigma^2 \Bbb RP^2$ is the spine of $V_2(\Bbb R^5)$, the Stiefel
manifold of $2$-frames in $\Bbb R^5$.  

If $X^n$ is a Poincar\'e complex with spine $K$ such that the top cell
of $X$ splits off after one suspension, then Theorem ~\ref{period-thm}
can be iterated to produce a sequence of Poincar\'e complexes $Y_j$ of
dimension $n+4j$ having spine $\Sigma^{2j} K$. In this way, we
obtain a periodic family of Poincar\'e complexes.  This motivates

\begin{defn} A finite complex $K$ is said to be 
{\it $j$-periodic} for some positive integer $j$ if there is an
integer $c$ and a sequence of Poincar\'e complexes $X_1,X_2, \dots$
such that the spine of $X_i$ is $\Sigma^{c+ij}K$. If $K$ is
$j$-periodic for some $j$, we say that $K$ is {\it periodic}. If $K$
is not periodic, we declare it to be {\it aperiodic}.

The {\it period} of
$K$, denoted $\text{period}(K)$, is the smallest positive integer
$r$ such that $K$ is $r$-periodic.
If there is no such $r$, then
we write $\text{period}(K) = \infty$.
\end{defn}

\begin{enumerate}
\item If $K$ is periodic, then $K$ is self Spanier-Whitehead
dual. This is a direct consequence of Prop.~\ref{B-S-criterion} below.
\item $\text{period}(S^k) = \infty$, since there are only a finite
number of Hopf invariant one elements.
\item $\text{period}(\Bbb RP^2) \le 2$,
because the spine of the Stiefel manifold $V_2(\Bbb R^{3+2i})$ (consisting of
two-frames in $\Bbb R^{3+2i}$) is $\Sigma^{2i}\Bbb RP^2$. Furthermore,
$\S \Bbb RP^2$ is the spine of $\text{SU}(3)/\text{SO}(3)$,
but Mahowald has pointed out to us 
that this is the only odd suspension of
$\Bbb RP^2$ which is the spine of a Poincar\'e complex.
So, $\text{period}(\Bbb RP^2) = 2$.
\item If $K$ is the spine of a Poincar\'e complex which embeds in 
codimension one, then Theorem ~\ref{cod-one} shows 
$\text{period}(K) = 1$.
\item If $K$ is the spine of a Poincar\'e complex whose top cell splits
off after a single suspension, then $\text{period}(K) \le 2$, by Theorem 
\ref{period-thm}.
\item Let $K$ be the spine of a $4k$-dimensional Poincar\'e complex
$X$, such that the Euler characteristic $\chi(K)$ is even.  Then
$\S K$ cannot be the spine of a Poincar\'e complex of dimension
$4k + 2$ by ~\ref{jims-ex}. Hence, $\text{period}(K) > 1$.
\item If $K = \text{spine}(X)$ and $L = \text{spine}(Y)$ for
Poincar\'e complexes $X^n$ and $Y^n$, then $K\vee L$ is periodic and
$
\text{period}(K\vee L) \, \, \le \text{lcm}(\text{period}(K),
\text{period}(L))\, .
$
 To see  this, 
set $r = \text{period}(K)$ and set $s = \text{period}(L)$.
Let $\ell$ denote their least common multiple. Define
$Z_i := X_{(i\ell)/r} \# Y_{(i\ell)/s}$, where $X_i$ has spine
$\Sigma^{ir} K$ and $Y_i$ has spine $\Sigma^{is}K$.
Then $Z_i$ has spine $\Sigma^{i\ell}(K \vee L)$.
Equality generally fails: e.g., $\text{period}(S^p) = \infty = \text{period}(S^q)$, but
$\text{period}(S^p \vee S^q) = 1$.
\end{enumerate}

Our notion of periodicity is linear, in that the gaps between the
number of suspensions of $K$ appearing in the definition is constant.
The following, due to Mahowald (private communication), is an example
of a $2$-cell complex which exhibits {\it exponential periodicity}, in
the sense that the gaps grow at an exponential rate.

\begin{mainthm}[Mahowald] Let $K = \Bbb HP^2$ be the homotopy cofiber of
the $\nu\:S^7 \to S^4$. Let $\delta(i) = 2^{i+2}$. 
Then $\Sigma^{\delta(i)} K$ is a spine of a Poincar\'e complex for $i>0$. 
Furthermore, one cannot fill in the gaps: if $\Sigma^jK$ 
is the spine of a Poincar\'e complex,
for some $j>0$, then $j = 2^{i+2}$ for some $i$.
\end{mainthm}

We ask a final question: When is a finite complex periodic?


\end{document}